\documentclass{amsart}
\usepackage{graphicx} 
\usepackage{amsmath}
\usepackage{amssymb}
\usepackage{amsthm}
\usepackage[all,arc]{xy}
\usepackage{graphicx}
\usepackage{mathrsfs}
\usepackage{tikz-cd}
\usepackage{xcolor}
\usepackage{mathtools}
\usepackage{url}

\theoremstyle{plain}
\newtheorem{theorem}{Theorem}
\newtheorem*{theoremB}{Theorem B}
\newtheorem{prop}{Proposition}
\newtheorem*{thmA}{Fact A}
\newtheorem{lem}{Lemma}

\theoremstyle{definition}
\newtheorem{defn}{Definition}

\newtheorem{remark}{Remark}
\newtheorem*{remark*}{Remark}

\numberwithin{theorem}{section}
\numberwithin{equation}{section}

\newcommand{\LL}{\mathscr{L}}
\newcommand{\Univ}{\mathscr{U}}
\newcommand{\rUniv}{\underline{{u}}}
\newcommand{\Prim}{\operatorname{Prim}}
\newcommand{\ad}{\operatorname{ad}}
\newcommand{\rLie}{\mathrm{rLie}}

\newcommand{\gr}{\operatorname{gr}}
\newcommand{\grM}{\operatorname{gr}^{\mathscr M}}
\newcommand{\grZ}{\operatorname{gr}^{{\mathcal Z}, p}}

\newcommand{\catlie}{\mathbf{Lie}}
\newcommand{\catReslie}{\mathbf{rLie}}
\newcommand{\catAlg}{\mathbf{Alg}}
\newcommand{\catHopf}{\mathbf{Hopf}}

\newcommand{\FF}{\mathbb{F}}
\newcommand{\Z}{\mathbb Z}

\newcommand{\kernel}{\mathrm{ker}}
\newcommand{\argu}{\hbox to 1.5ex{\hrulefill}}  
\newcommand{\qt}[1]{``{#1}''}

\title{Presenting the Zassenhaus Lie algebra by the Magnus Lie algebra}
\author{Ettore Marmo, David Riley and Thomas Weigel}
\date{\today}

\keywords{Magnus Lie algebra associated to a group, restricted  Zassenhaus Lie algebra associated to a group, Jennings filtration, graded restricted Lie algebra, restrictification functor}
\subjclass[2010]{20F40; 17B70, 20F14}

\begin{document}
\setcounter{page}{1}

\begin{abstract}
It is shown that the Zassenhaus restricted $\FF_p$-Lie algebra of a (pro-$p$) group $G$ can be presented by the Magnus Lie algebra of $G$. For the class of (pro-$p$) groups for which the terms of the lower central series are torsion-free, 
 the
Zassenhaus restricted $\FF_p$-Lie algebra can be explicitly computed from the Magnus Lie algebra. 
These results apply to orientable surface groups, right-angled Artin groups, pure braid groups, fundamental groups of supersolvable hyperplane arrangements and fundamental groups of strictly supersolvable toric arrangements.
\end{abstract}


\maketitle

\section{Introduction}


In 1937, 
N. Jacobson introduced in \cite{Jac37} the notion of a \emph{restricted $\FF$-Lie algebra} for fields $\FF$ with positive characteristic. Soon after (cf. \cite{Zas39}), H. Zassenhaus  established a procedure to associate to any group $G$, an $\mathbb N$-graded, restricted $\FF_p$-Lie algebra $\grZ_\bullet(G)$, where $\FF_p$ denotes the finite field of cardinality $p$. In case $G$ is a pro-$p$ groups, the homogeneous components $\grZ_k(G)$ are compact $\FF_p$-vector spaces. Moreover, for a finitely generated, discrete group $G$ one has:
\begin{equation}\label{eq:grz.pro-p.completion}
    \grZ_\bullet(G) \cong \gr^{\mathcal Z}_\bullet(\widehat G_{p})
\end{equation}
where $i_G : G \to \widehat G_{p}$ denotes the pro-$p$ completion of $G$. From now on, for a pro-$p$ group $G$ we omit the ``$p$" 
in the notation of the restricted Zassenhaus $\FF_p$-Lie algebra $\gr^{\mathcal Z}_\bullet(G)$.
\noindent
 Almost simultaneously, Jennings gave a description of the homogeneous terms of this restricted $\FF_p$-Lie algebra in terms of combinatorial data of the group $G$, thus showing that it can be defined in a similar way as the Magnus $\Z$-Lie algebra $\grM_\bullet(G)$ established some years earlier (cf. \cite{Mag37} \cite{Mag40}). 
\noindent
In general, for a given group $G$ the Magnus $\Z$-Lie algebra is easier to compute, however some applications, like the May spectral sequence (see \cite{Day94} \cite{Leo24}), require knowing  
the isomorphism type of the restricted Zassenhaus $\FF_p$-Lie algebra $\grZ_\bullet(G)$ explicitly. 
\noindent
The main purpose of this note is to provide a description of the restricted Zassenhaus $\FF_p$-Lie algebra $\grZ_\bullet(G)$ in terms of the Magnus $\Z$-Lie algebra $\grM_\bullet(G)$. 
This description uses the restrictification functor 
\begin{equation}\label{eq:restrictification.functor}
    (-)^{[p]} : \FF_p\text{-}\catlie \longrightarrow \FF_p\text{-}\catReslie
\end{equation}
which is the left adjoint of the forgetful functor $f :\FF_p\text{-}\catReslie \longrightarrow\ \FF_p\text{-}\catlie$.
In this setting one obtains the following presentation of the restricted Zassenhaus $\FF_p$-Lie algebra of a group $G$.
\begin{thmA}
    Let $G$ be a group (resp. pro-$p$ group). There exists a natural surjective homomorphism of $\mathbb N$-graded, $1$-generated, restricted $\FF_p$-Lie algebras
    \begin{equation}\label{eq:pres}
        \hat\theta_{G, \bullet} : (\grM_\bullet(G)\otimes \FF_p)^{[p]} \twoheadrightarrow \grZ_\bullet(G) .
    \end{equation}
\end{thmA}
\noindent
From now on we refer to the kernel of $\hat\theta_{G, \bullet}$ as the \emph{Magnus-Zassenhaus kernel}.
In the following extreme cases, this kernel can be described explicitly showing that knowledge about the restricted Zassenhaus $\FF_p$-Lie algebra can be translated into knowledge about the Magnus-Zassenhaus kernel (see Section \S\ref{section 2.3}
and Subsection~\S\ref{sss:unif}).
\begin{remark}
\label{cor1}
    If $G$ is a pro-$p$ group of exponent $p$, then $\gr^{\mathcal Z}_\bullet(G)$ is an $\mathbb N$-graded, restricted $\FF_p$-Lie algebra with zero $p$-map. In particular, $\kernel(\hat\theta_{G, \bullet})$ consists of the restricted $\FF_p$-Lie ideal of 
    $(\grM_\bullet(G)\otimes \FF_p)^{[p]}$  generated by the elements $x^{[p]}$ where $x \in \grM_\bullet(G)\otimes \FF_p$.
\end{remark}
\noindent
For a uniformly powerful pro-$p$ group $G$ one has $D_n(G) = G^{p^n}$ (cf. \cite{DDSMS}). In particular, $\gr_\bullet^{\mathcal Z}(G)$ is an $\mathbb N$-graded abelian restricted $\FF_p$-Lie algebra with injective $p$-map. From this, one concludes the following.
\begin{remark}
\label{cor2}
    If $G$ is a uniformly powerful pro-$p$ group, then $\gr^{\mathcal Z}_\bullet(G)$ is abelian. In particular,  $\kernel(\hat\theta_{G, \bullet})$ consists of the restricted $\FF_p$-Lie ideal of $(\grM_\bullet(G)\otimes \FF_p)^{[p]}$ generated by the elements of $\grM_{\geq 2}(G)\otimes \FF_p$. 
\end{remark}

A finitely generated discrete group (resp. pro-$p$ group) $G$ will be said to be \emph{$\gamma$-free} if $G/\gamma_k(G)$ is torsion-free for all $k \geq 1$, where $\gamma_k(G)$ denotes the (closure of the) $k$-th term of the lower central series of $G$.  

\begin{theoremB}
    Let $G$ be a f.g. $\gamma$-free group (resp. $\gamma$-free pro-$p$ group). Then, the homomorphism $\hat\theta_{G, \bullet}$ (cf. \eqref{eq:pres}) is an isomorphism.
\end{theoremB}
\noindent
Theorem B applies in many significant cases. 
\begin{itemize}
    \item Right-angled Artin groups; \cite{BHS20} \cite{PS06}
    \item certain types of one-relator (pro-$p$) groups including surface groups; \cite{Lab70} 
    \item almost-direct products of $\gamma$-free groups; \cite{FR85} 
    \item Magnus groups; \cite{Bau67}, \cite{Shm70}
    \item pure braid groups; \cite{Koh85} 
    \item more generally, fundamental groups of supersolvable complex hyperplane arrangement \cite{FR85} and strictly supersolvable toric arrangements \cite{BD24}
\end{itemize}

Theorem B states that all finitely generated $\gamma$-free pro-$p$ groups have trivial Magnus-Zassenhaus kernel $\kernel(\hat{\theta}_{G,\bullet})$. However, the class of 
finitely generated pro-$p$ groups with trivial Magnus-Zassenhaus kernel $\kernel(\hat{\theta}_{G,\bullet})$ is strictly larger than the class of finitely generated $\gamma$-free pro-$p$ groups. (cf. Subsection~\ref{sss:mild}).

\subsection*{Acknowledgements}
The first author was supported by the grant $\text{PCI2024-155096-2}$ financed by the Ministerio de Ciencia, Innovaci\'on y Universidades (MICIU), Agencia Estatal de Investigaci\'on (AEI) and co-financed by the European Union. 

\section{Lie algebras}
\label{s:Liealg}
In this section we recall some basic constructions from the theory of Lie algebras and restricted Lie algebras.

\subsection{Universal envelopes and Hopf algebras}
\label{ss:UniHopf}
Let $\FF$ be a field. We will denote by $\FF\text{-}\catAlg$ the category of associative, unital $\FF$-algebras with $\FF$-algebra homomorphisms and similarly by $\FF\text{-}\catlie$ the category of $\FF$-Lie algebras with $\FF$-Lie algebra homomorphisms. These two categories are related by a well known pair of adjoint functors:
\begin{equation}\begin{tikzcd}
    \FF\text{-}\catlie \arrow[rr, yshift=3pt, bend left=10pt, "\Univ"] &  {\bot} & \FF\text{-}\catAlg \arrow[ll, yshift=-3pt, bend left=10pt, "\mathcal L" below]
\end{tikzcd}\end{equation}
The right adjoint assigns to each associative $\FF$-algebra $(A,\, \cdot)$ the $\FF$-Lie algebra ${\mathcal L}(A) = (A,\, [\text{-},\text{-}]_A)$ whose bracket is given by $[a, b]_A := a\cdot b - b \cdot a$ for all $a, b \in A$. The left adjoint assigns to each $\FF$-Lie algebra $\LL$ its \emph{universal enveloping algebra} $\Univ(\LL)$,
which can be presented as a quotient of $T^\bullet(\LL)$, the free tensor $\FF$-algebra over $\LL$, modulo the ideal $$J_\LL = \langle x \otimes y - y \otimes x - [x, y] \mid x, y \in \LL \rangle.$$ The classical Poincar\'e-Birkoff-Witt theorem (commonly referred to as the ``PBW theorem'') states that natural map, $\iota_\LL : \LL \to \Univ(\LL)$, is an injective morphism of $\FF$-vector spaces and that $\Univ(\LL)$ is generated as an $\FF$-algebra by the image of $\iota_\LL$. In order to simplify notation we will identify a Lie algebra $\LL$ with its image $\iota_\LL(\LL)$ in its universal enveloping algebra $\Univ(\LL)$.

The universal enveloping algebra $\Univ(\LL)$ inherits from $T^\bullet(\LL)$ the structure of an $\FF$-Hopf algebra, thus the functor $\Univ$ actually takes values in the category $\FF\text{-}\catHopf$. 
If $H$ is a Hopf algebra over $\FF$, an element $x \in H$ is said to be \emph{primitive} if it satisfies the identity $\Delta(x) = x \otimes 1 + 1 \otimes x$ where $\Delta : H \to H \otimes H$ is the coproduct of $H$. 
\begin{remark}\label{rem:prim}
    The $\FF$-subspace $\Prim(H)$ of primitive elements is closed under the commutator bracket $[x, y]_H := xy - yx$ and thus is an $\FF$-Lie subalgebra of $H$. \medskip \\
    In case $\FF$ has positive characteristic $p$, for every $x \in \Prim(H)$ one has:
    \begin{equation}
        \Delta(x^p)=\Delta(x)^p=(1\otimes x+x\otimes 1)^p=
        \sum_{k=0}^p {p\choose k} x^k\otimes x^{p-k}=1\otimes x^p+x^p\otimes 1.
    \end{equation}
    Hence $x^p \in \Prim(H)$, i.e., $\Prim(H)$ is  a restricted $\FF$-Lie algebra.
\end{remark}

The assignment $H \mapsto \Prim(H)$ extends to a functor $\Prim : \FF\text-\catHopf \longrightarrow \FF\text-\catlie$ that is right adjoint to the universal enveloping (Hopf) algebra functor:
\begin{equation}\begin{tikzcd}\label{eq:adj.hopf}
    \FF\text{-}\catlie \arrow[rr, yshift=3pt, bend left=10pt, "\Univ"] & {\bot} & \FF\text{-}\catHopf \arrow[ll, yshift=-3pt, bend left=10pt, "\Prim" below] \mathclap{\quad \quad \quad}
\end{tikzcd}\end{equation}

An $\mathbb N$-\emph{graded} Lie algebra, where $\mathbb N$ is the set of positive integers, is a Lie algebra $\LL_\bullet$ equipped with a direct sum decomposition $\LL_\bullet = \bigoplus_{n \geq 1} \LL_n$ compatible with the Lie bracket, i.e. for all $n, m \in \mathbb N$ one has $[\LL_n, \LL_m] \subseteq \LL_{n+m}$.
\noindent
The free tensor algebra of an $\mathbb N$-graded Lie algebra $\LL_\bullet$ (or more generally for an $\mathbb N$-graded vector space) is canonically $(\mathbb N_0 \times \mathbb N_0)$-graded:
\begin{equation}
	T^{n, k}(\LL_\bullet) = \bigoplus_{i_1 + i_2 + \cdots + i_k = n} \LL_{i_1} \otimes \LL_{i_2} \otimes \cdots \otimes \LL_{i_k}
\end{equation}
where $\mathbb N_0 = \mathbb N\cup \{0\}$ thus one has $T^{0, 0}(\LL_\bullet) = \FF \cdot {\rm id}$ and $T^{n, 0}(\LL_\bullet) = 0$ for $n \geq 1$. Note that the ideal $J_{\LL_\bullet}$ is a graded ideal with respect to the grading by the first component, that is, $J_{\LL_\bullet} = \bigoplus_{n \geq 1} J_{\LL_\bullet}^n$ where $J_{\LL_\bullet}^n = J_{\LL_\bullet} \cap T^{n, \bullet}(\LL_\bullet)$ and $T^{n, \bullet}(\LL_\bullet) = \bigoplus_{k \geq 0} T^{n, k}(\LL_\bullet)$ for all $n \geq 1$. Thus, the universal enveloping algebra $\Univ(\LL_\bullet)$ inherits a grading as well.

\subsection{Restricted Lie algebras and restrictification} 
Let $\FF$ be a field of characteristic $p > 0$
\begin{defn}\label{def:restricted.lie.alg}
    Let $\LL$ be an $\FF$-Lie algebra. A \emph{$p$-map} for $\LL$ is a function $(-)^{[p]} : \LL \to \LL$ satisfying:
    \begin{enumerate}
        \item[($\rm rL_1$)] $(\lambda x)^{[p]} = \lambda^p x^{[p]}$;
        \item[($\rm rL_2$)] $(x + y)^{[p]} = x^{[p]} + y^{[p]} + \sum_{k=1}^{p-1} \frac{s_k(x, y)}{k}$;
        \item[($\rm rL_3$)] $[x^{[p]}, y] = \ad(x)^p(y)$;
    \end{enumerate}
    where $s_k(x, y)$ is the coefficient of $t^{k-1}$ in the formal expression $\ad(tx+y)^{p-1}(x)$ and $\ad(x)(y) = [x, y]$. An $\FF$-Lie algebra equipped with a $p$-map is said to be a \emph{restricted} $\FF$-Lie algebra.
\end{defn}

A consequence of the PBW theorem, one knows that there is an injective map of $\FF$- of $\LL$ 
into its universal enveloping algebra $(\Univ(\LL), [\text{-},\text{-}]_{\Univ(\LL)})$. Since $\Prim$ is the right adjoint of $\Univ$, $\LL$ is also an $\FF$-Lie subalgebra of $\Prim(\Univ(\LL))$. It is well know that for fields of characteristic $0$ the two Lie algebras $\LL$ and $\Prim(\Univ(\LL))$ coincide (see \cite{Bou89} \cite{MM65}), while for fields of characteristic $p > 0$, $\Prim(\Univ(\LL))$ is strictly larger than $\LL$. By Remark \ref{rem:prim} one has that
\begin{equation}
    \LL^{[p]} = \Prim(\Univ(\LL))
\end{equation}
is a restricted Lie algebra which will be called the  \emph{canonical restrictification of $\LL$}.
\begin{remark}\label{rem:Lp.in.UL}
    Note that by \cite[Lemma 1, p. 157 and Lemma 4 p. 189]{Jac62} $\LL^{[p]}$ coincides with the restricted $\FF$-Lie subalgebra of $\Univ(\LL)$ generated by $\LL$.
\end{remark}
\noindent
Restricted Lie algebras together with Lie algebra morphisms which respect the $p$-power maps form a category, $\FF\text{-}\catReslie$, and the assignment $\LL \mapsto \LL^{[p]}$ extends to a covariant functor $(-)^{[p]} : \FF\text{-}\catlie \to \FF\text{-}\catReslie$ which is left adjoint to the obvious forgetful functor from $\FF\text{-}\catReslie$ to $\FF\text{-}\catlie$.

Let $(A, \cdot)$ be an associative $\FF$-algebra, then the Lie algebra ${\mathcal L}(A)$ equipped with the map $x \mapsto x^p$, defined using the associative product of $A$ is a restricted Lie algebra. As in the case of Lie algebras, this defines a functor ${\mathcal L}_p : \FF\text-\catAlg \to \FF\text-\catReslie$. As in the general case, every restricted $\FF$-Lie algebra $\LL$ has a restricted universal enveloping algebra $\rUniv(\LL)$ which can be either defined as the quotient of $\Univ(\LL)$ by the ideal generated by $x^{[p]}-x^p$ or as the left adjoint of the functor ${\mathcal L}_p$ (see \cite{Jac37}).
The just mentioned properties can be summarized as follows (for further details see \cite{MacL}).
\begin{prop}\label{prop:adj.restrictification}
    	The following functors form adjoint pairs:
	\begin{equation}\label{eq:adjoint.1}\begin{tikzcd}
    \FF\text-\catReslie \arrow[rr, yshift=3pt, bend left=10pt, "\rUniv"] & {\bot} & \FF\text-\catAlg \arrow[ll, yshift=-3pt, bend left=10pt, "{\mathcal L}_p" below]
    \end{tikzcd}\end{equation}

	\begin{equation}\label{eq:adj.restr}\begin{tikzcd}
        \FF\text-\catlie \arrow[rr, yshift=3pt, bend left=10pt, "(-)^{[p]}" xshift=4pt] & {\bot} & \FF\text-\catReslie \arrow[ll, yshift=-3pt, bend left=10pt, "\rm f" below]
    \end{tikzcd}\end{equation}
	where $\rm f$ is the forgetful functor. 
\end{prop} 
\begin{proof}
    The natural isomorphism \eqref{eq:adjoint.1} is a direct consequence of \cite[Definition 1, p. 152]{Jac62}. 
    In order to prove \eqref{eq:adj.restr}, consider an $\FF$-Lie algebra $\LL$, a restricted $\FF$-Lie algebra $\mathcal H$ and a homorphism $\phi : \LL \to \mathcal H$ of $\FF$-Lie algebras. This induces a unique homomorphism of associative $\FF$-algebras $\Univ(\phi):\Univ(\LL) \to \Univ(\mathcal H)$ (see \cite[Theorem 1, p. 152]{Jac62}). Let $\pi$ be the canonical projection to $\rUniv(\mathcal H)$, the restricted universal enveloping algebra of $\mathcal H$. Now as $\overline{\phi} = \pi\circ\Univ(\phi)$ is a homomorphism of $\FF$-Hopf algebras, it maps $\Prim(\Univ(\LL))$ to $\Prim(\rUniv(\mathcal H))$ which coincides with $\mathcal H$. This yields canonical one-to-one correspondences between 
   $\hom_{\FF\text-\catlie}(\LL, {\rm f}(\mathcal H))$ and $\hom_{\FF\text-\catReslie}(\LL^{[p]}, \mathcal H)$.    
\end{proof}

\begin{remark}\label{rem:pbw.restricted}
If $\LL_\bullet$ is a graded Lie algebra, then its restrictification $\LL^{[p]}$ admits a special grading. Consider the following subspaces of $\LL^{[p]}$ 
\begin{equation}\label{eq:pbw.restrictification}
	\LL^{[p]}_n = \bigoplus_{\substack{i, j \geq 0 \\ ip^j = n}} (\LL_i)^{[p]^j}
\end{equation}
where $(\LL_i)^{[p]^j}$ denotes the subspace of $\LL^{[p]}$ spanned by elements of the form $a_i^{[p]^j}$ with $a \in \LL_i$. From the PBW theorem we deduce that the subspaces $\LL^{[p]}_n$ are all disjoint. Moreover since $\LL^{[p]}$ is spanned by the elements of $\LL$ and the images of the $p$-power map, we have that $\LL^{[p]} = \bigoplus_{n \geq 0} \LL_n^{[p]}$.
\medskip \\
In particular, by \cite[Lemma 4, p. 189 and Theorem 12, p. 191]{Jac62}, if $\mathcal B \subset \LL$ is an $\FF$-basis for $\LL$, then 
    \begin{equation}
        \mathcal B^{[p]} = \left\{ u^{[p]^k}\mid k \in \mathbb N, u \in \mathcal B \right\}
    \end{equation}
is an $\FF$-basis for $\LL^{[p]}$.
\end{remark}

\noindent
We record the following observations.

\begin{prop}
    For an $\FF$-Lie algebra $\LL$ where $\FF$ is a field of positive characteristic $p$, one has $\rUniv(\LL^{[p]}) \cong \Univ(\LL)$
\end{prop}
\begin{proof}  
    Recall that $\LL^{[p]} = \Prim(\Univ(\LL))$. So $\LL^{[p]}$ is a restricted $\FF$-Lie subalgebra of the universal enveloping algebra $\Univ(\LL)$. 
    By the universal property of the restricted universal envelope, one has a homomorphism of associative $\FF$-algebras $\phi : \rUniv(\LL^{[p]})\to \Univ(\LL)$. As $\LL$ is contained in $\LL^{[p]}$, the same argument yields a homomorphism of associative $\FF$-algebras $\psi : \Univ(\LL)\to\rUniv(\LL^{[p]})$.
    Now, as $\psi \circ \phi = {\rm id}_{\LL^{[p]}}$ and $\phi \circ \psi = {\rm id}_{\LL}$, one is the inverse of the other.
\end{proof}

\subsection{$\mathbb N$-graded restricted $\FF_p$-Lie algebras with zero $p$-map}\label{section 2.3}
Let $(\LL, [\cdot,\cdot], (-)^{[p]})$ be a restricted $\FF_p$-Lie algebra such that $x^{[p]} = 0$ for all $x \in \LL$. By Definition \ref{def:restricted.lie.alg} (${\rm rL}_2$) one has that $s_k(x, y) = 0$ for all $1 \leq k \leq p-1$, in particular $\LL$ satisfies the $(p-1)$-Engel condition, that means ${\rm ad}^{p-1}(x) = 0$ for all $x \in \LL$. Conversely, 
if $\LL$ is an $\mathbb N$-graded $\FF_p$-Lie algebra satisfying the $(p-1)$-Engel condition, then $s_k(x, y) = 0$ for all $1 \leq k \leq p-1$, in particular together with the zero map $\LL$ constitutes a restricted $\FF_p$-Lie algebra.
\noindent
By \eqref{eq:adj.restr} there exists a homomorphism of restricted $\FF_p$-Lie algebras $j_\LL^{[p]} : \LL^{[p]} \to \LL$ corresponding to the identity ${\rm id}_\LL$ under the adjunction. The kernel of $j_\LL^{[p]}$ is the restricted $\FF_p$-Lie ideal generated by $x^{[p]}$ for all $x \in \LL$.

\section{Subgroup series and associated graded Lie algebras}

In this section we study the Lie algebras associated to two important descending series of subgroups for a group $G$.

\subsection{The lower central series}
The \emph{lower central series} of a group $G$ is the series of characteristic subgroups $\{\gamma_i(G)\}_{i \geq 1}$ of $G$ defined recursively by
\begin{align*}
    \gamma_1(G) &= G \\
    \gamma_{i+1}(G) &= [G, \gamma_i(G)], \quad \text{ for  $i \geq 1$.} 
\end{align*}
\noindent
The lower central series is an example of a (descending) strongly central series, 
i.e. it satisfies $\gamma_{i+1}(G) \leq \gamma_i(G)$ for all $i \geq 1$ and $[\gamma_i(G), \gamma_j(G)] \leq \gamma_{i+j}(G)$ for all $i, j \geq 1$.
To this series one may associate the graded $\Z$-Lie algebra: 
\[
	\grM_\bullet(G) = \bigoplus_{n \geq 1} \grM_n(G)
\]
where $\grM_n(G) := \gamma_n(G)/\gamma_{n+1}(G)$ which is also called the \emph{Magnus $\Z$-Lie algebra} associated to $G$. For $i, j \geq 1$, for $x \in \gamma_i(G)$ and $y \in \gamma_j(G)$
the Lie bracket:
\begin{align*}
	[-, -] : \grM_i(G) \otimes_\Z \grM_j(G) &\to \grM_{i+j}(G)  \quad \intertext{is defined by} 
	[x\gamma_{i+1}(G), \, y \gamma_{j+1}(G)] &= [x, y]\gamma_{i+j+1}(G).
\end{align*}
\noindent
From the definition one deduces easily that it is $\Z$-bilinear and alternating. The Jacobi identity can be derived from the well-known Hall-Witt formula for group commutators.

\subsection{Magnus Lie algebras of pro-$p$ completions}
Let $G$ be a finitely generated nilpotent group.
As $\Z_p$ is a binomial lambda ring (see \cite[Definition 1.10]{Yau10}), there is a category of $\Z_p$-powered nilpotent groups introduced by P. Hall in \cite[§6]{Hall69}. 
In particular there exists a canonical map 
\begin{equation}
    h_{G, \Z_p} : G\longrightarrow G\otimes{\Z_p}
\end{equation}
from the nilpotent group $G$ to the $\Z_p$-powered nilpotent group $G \otimes {\Z_p}$ which is characterized by a universal property. 
This map shows that the   functor $-\otimes \Z_p $ is the left adjoint to the forgetful functor from the category of $\Z_p$-powered nilpotent groups to the category of nilpotent groups.
Moreover 
\begin{equation}
    \grM_\bullet(G \otimes \Z_p)\cong \grM_\bullet(G) \otimes \Z_p.
\end{equation}
\noindent
If $G$ is a finitely generated nilpotent group, and $g_1, \dots, g_n$ are elements in $G$ such that $G = \prod_{i = 1}^n g_i^\Z$, then $G \otimes \Z_p = \prod_{i = 1}^n g_i^{\Z_p}$. Hence $G \otimes \Z_p$ is a pro-finite group in which every element generates a cyclic pro-$p$ group, i.e. $G\otimes \Z_p$ is a pro-$p$ group.
Thus, $h_{G, \Z_p}(G)$ coincides with the pro-$p$ completion $\widehat G_p$ of $G$. This shows that
\begin{equation}\label{eq:grM.comp}
    \grM_\bullet(\widehat G_p)\cong \grM_\bullet(G)\otimes \Z_p
\end{equation}
holds for all finitely generated nilpotent groups and thus for all finitely generated groups.

\subsection{Dimension subgroups} Let $G$ be a group and let $R$ be a commutative unital ring. The \emph{augmentation ideal} of the group ring $RG$ is the kernel $\omega_R(G)$ of the augmentation map $\varepsilon : RG \to R$ defined by sending each group element $g \in G$ to the unit $1 \in R$. The augmentation ideal coincides with the ideal generated by elements of $RG$ of the form $1-g$ with $g \in G$.  

\begin{defn}
    For a group $G$ and ring $R$, we define the $n$-th \emph{dimension subgroup of $G$ over $R$} as:
    \[
        D_n(G, R) := \{g \in G \mid 1-g \in \omega_R(G)^n\}
    \]
\end{defn}
\noindent
It can be shown that $D_n(G, R)$ is a subgroup of $G$ for every $n \geq 1$. \\
For a long time, the dimension subgroup series over $\Z$ and the lower central series were believed to coincide, this was known as the ``Dimension Subgroup Conjecture''. While it is true that, for any group $G$, one has $D_2(G, \Z) = \gamma_2(G)$ and even $D_3(G, \Z) = \gamma_3(G)$ (a deep theorem by Higman and Rees, cf. \cite{Pas79}, \cite{Gup90}), there exists a group $G$ such that $D_4(G, \Z)/\gamma_4(G)$ is nontrivial. This counterexample to the Dimension Subgroup Conjecture was constructed by Rips in \cite{Rip72}.

\medskip
For $R = \FF_p$, the dimension series $\{D_n(G, \FF_p)\}_{n \geq 1}$, is called the \emph{$p$-Zassenhaus} series of $G$ (also known as Jennings-Lazard-Zassenhaus series) and the associated restricted Lie algebra is denoted by $\grZ_\bullet(G)$. 

This series of subgroups is not only a strongly central series but it is also a \emph{$N_p$-central series}, i.e. it satisfies the additional property that $D_n(G, \FF_p)^p \leq D_{pn}(G, \FF_p)$ for all $n \geq 1$ (cf. \cite[p. 85]{Pass85}). Similar 
to the lower central series, the $p$-Zassenhaus series is the fastest descending $p$-central series.  
This additional property implies that the associated graded Lie algebra $\grZ_\bullet(G)$ equipped with a $p$-map induced by the group theoretic $p$-power map $x \mapsto x^p$, is a restricted Lie algebra over $\mathbb F_p$  (see \cite{Zas39}, \cite{Laz54}).
\medskip\\
A more explicit description of the $p$-Zassenhaus series was given by Jennings in \cite{Jen41}. In the following, we will fix a prime number $p$ and denote $D_n(G, \FF_p)$ simply by $D_n(G)$. 
\begin{theorem}[\cite{Jen41}, \cite{Laz54}]\label{th:jennings}
    Let $G$ be a group. Then for every $n \geq 1$ the subgroup $D_n(G)$ can be expressed as follows
    \begin{equation}\label{eq:dimension.jennings}
        D_n(G) = \prod_{\substack{i, j \geq 0 \\ ip^j \geq n}} \gamma_i(G)^{p^j}
    \end{equation}
    where $\gamma_i(G)^{p^j}$ is the subgroup generated by the $p^j$-th powers of elements of $\gamma_i(G)$.
\end{theorem}
One can refine  
the $p$-Zassenhaus series by defining, for $1 \leq k \leq n$, the subgroups:
\begin{equation}\label{eq:dimension.jennings.refined}
    D_{n, k}(G) := \prod_{\substack{i \geq k,\, j \geq 0 \\ ip^j \geq n}} \gamma_i(G)^{p^j}
\end{equation}

\begin{remark}\label{rem:explicit.filtration}
    Let $n, i \geq 1$ and define $j(n, i)$ to be the least positive integer $j$ such that $ip^{j} \geq n$. Then, for all $i \leq k \leq n$, one has:
    \begin{equation}                        D_{n, k}(G) = \prod_{i=k}^n \gamma_i(G)^{p^{j(n, i)}} 
    \end{equation}
    in particular, the following inclusions hold:
    \[
        D_n(G) = D_{n, 1}(G) \geq D_{n, 2}(G) \geq \cdots \geq D_{n, n}(G) = \gamma_n(G).
    \] 
\end{remark}

\begin{prop}\label{prop:hall}
    Let $G$ be a group, $i \geq 1$ and $j \geq 0$, the following hold:
    \begin{enumerate}
        \item $(xy)^{p^j} \equiv x^{p^j} y^{p^j} \mod D_{ip^j, i+1}(G)$ for every $x, y \in \gamma_i(G)$;
        \item $[x^{p^j}, y] \in D_{ip^j, i+1}(G)$ for every $x \in \gamma_i(G)$ and $y \in G$.
    \end{enumerate}
\end{prop}
\begin{proof}
    The classic Hall-Petrescu formulae \cite[I, §3 9.4]{Hup67} state the following relations
    \begin{equation}\label{eq:hall1}
        (xy)^{p^j} \equiv x^{p^j} y^{p^j} \mod \gamma_2(H)^{p^j} \prod_{r=1}^j \gamma_{p^r}(H)^{p^{j-r}}
    \end{equation}
    and 
    \begin{equation}\label{eq:hall2}
        [x^{p^j}, y] \equiv [x, y]^{p^j} \mod \gamma_2(K)^{p^j} \prod_{r=1}^j \gamma_{p^r}(K)^{p^{j-r}}
    \end{equation}
    where $H = \langle x, y \rangle$ and $K = \langle x, [x, y] \rangle$. Notice that when $x, y \in \gamma_i(G)$ we have  $\gamma_k(H) \leq \gamma_{ki}(G)$ for each $k\ge 1$.  
    Hence, $\gamma_2(H)^{p^j} \leq \gamma_{2i}(G)^{p^j} \leq D_{2ip^j,2i}(G)$, while for $r \in{1, \dots, j}$ we have $\gamma_{p^r}(H)^{p^{j-r}} \leq \gamma_{ip^r}(G)^{p^{j-r}}\leq D_{ip^j,ip^r}(G)$.  Since these subgroups all belong to $D_{ip^j, i+1}(G)$, \eqref{eq:hall1} proves (1). \medskip \\
    The second assertion is proven similarly using \eqref{eq:hall2} and observing that if $x \in \gamma_i(G)$ then $[x, y]^{p^j} \in \gamma_{i+1}(G)^{p^j} \leq D_{ip^j, i+1}(G)$ as well.
\end{proof}

\subsection{A canonical homomorphism of graded Lie algebras}
Since the $p$-Zassenhaus series is a central series, for any group $G$ one has that $\gamma_n(G) \leq D_n(G)$ for all $n \geq 1$. So, for each $n \geq 1$ one may define a homomorphism of abelian groups 
\begin{equation}
    \rho_{G, n} : \grM_n(G) \to \grZ_n(G)
\end{equation}
given by $\rho_{G, n}(x\gamma_{n+1}(G)) := xD_{n+1}(G)$. The homomorphisms $\rho_{G, n}$ give rise to a morphism $\rho_{G, \bullet} : \grM_\bullet(G) \to \grZ_\bullet(G)$ of graded abelian groups, which turns out to be a map of graded $\Z$-Lie algebras. 

Since $\grZ_\bullet(G)$ is naturally a graded Lie algebra over $\FF_p$, we will consider the induced homomorphism of $\FF_p$-Lie algebras: 
\begin{equation}
    \theta_{G, \bullet} = (\rho_{G, \bullet} \otimes {\rm id}_{\FF_p}) : \grM_\bullet(G) \otimes \FF_p \longrightarrow \grZ_\bullet(G)
\end{equation}
given by $\theta_{G, n}(x \gamma_{n+1}(G) \otimes 1) := \rho_n(x\gamma_{n+1}(G)) = xD_{n+1}(G)$ for each $x \in \gamma_n(G)$ and $n > 0$. From Proposition \ref{prop:adj.restrictification} \eqref{eq:adj.restr} 
one deduces the existence of a natural homomorphism of restricted $\FF_p$-Lie algebras:
\begin{equation}\label{eq:nat.restr}
    \hat \theta_{G, \bullet} : (\grM_\bullet(G) \otimes \FF_p)^{[p]} \longrightarrow \grZ_\bullet(G)
\end{equation}
\noindent
As $\hat{\theta}_{G, 1}$ is an isomorphism by construction, and $\grZ_\bullet(G)$ is a $1$-generated $\mathbb N$-graded restricted $\FF_p$-Lie algebra, one concludes the following:

\begin{lem}\label{lem:hat.theta.epi}
	The map $\hat \theta_{G, \bullet}$ is surjective.
\end{lem}

\section{$\gamma$-free groups}
In this section, we study the properties of the map $\hat\theta_\bullet$ for a particular class of (pro-$p$) groups.

\begin{defn}
    A finitely generated (pro-$p$) group $G$ is said to be $\gamma$-free if $\gamma_n(G)/\gamma_{n+1}(G)$ is a free abelian (pro-$p$) group for all $n \geq 1$.
\end{defn}

The property of a group being $\gamma$-free can be stated equivalently in terms of its nilpotent quotients.

\begin{lem}\label{lem:gamma.free.equiv}
    For $G$ a finitely generated (pro-$p$) group, the following are equivalent:
    \begin{itemize}
        \item[(i)] $G$ is $\gamma$-free;
        \item[(ii)] $G/\gamma_n(G)$ is $\gamma$-free for all $n \geq 1$; and
        \item[(iii)] $G/\gamma_n(G)$ is torsion-free for all $n \geq 1$. 
    \end{itemize}
\end{lem}

\begin{remark}\label{rem:nilpotent.quotients}
	The associated graded Lie algebras of a group $G$ and its nilpotent quotients $G/\gamma_c(G)$ are closely related. Indeed since for all $1 \leq n \leq c$ we have $\gamma_n(G/\gamma_c(G)) = \gamma_n(G)/\gamma_c(G)$, then 
\[
	\grM_n(G/\gamma_c(G)) \cong \begin{cases}
		\grM_n(G) & \text{if $1 \leq n < c$}\\
		0 & n \geq c
\end{cases}
\]

\end{remark}

\begin{lem}\label{lem:Dn.itsc.gamma}
    Let $G$ be a $\gamma$-free (pro-$p$) group. Then, for any $n \geq 1$ and $1 \leq k \leq n-1$ one has
    \begin{equation}
        \gamma_k(G) \cap D_n(G) = D_{n, k}(G)
    \end{equation}
\end{lem}
\begin{proof}
    The inclusion $D_{n, k}(G) \leq \gamma_k(G) \cap D_n(G)$ follows from the definitions. \\
    Let $g \in D_n(G)$, then $g$ can be expressed as a product
    \begin{equation}
        g = \prod_{i =1}^n g_i,
    \end{equation}
    where $g_i \in \gamma_i(G)^{p^{j(n, i)}}$ and, as in Remark \ref{rem:explicit.filtration}, $j(n, i)$ denotes the smallest integer $j \geq 0$ such that $ip^j \geq n$.
    Up to regrouping the terms if necessary, we may assume that either $g_i = 1$ or $g_i \neq 1$ and $g_i \notin \gamma_{i+1}(G)$.
    We prove by induction on $1 \leq k < n$ that $g \in \gamma_k(G)\cap D_n(G)$ implies that $g_i = 1$ for all $i \in \{1, \dots, k-1\}$, that is, $\gamma_k(G) \cap D_n(G) \leq D_{n, k}(G)$. For $k = 1$ there is nothing to prove. 
    Let $g \in \gamma_2(G) \cap D_n(G)$, then
    \begin{equation}
        g = g_1 \cdot \prod_{i=2}^{n} g_i = g_1 \cdot z.
    \end{equation}
    As $z \in \gamma_2(G)$, one must have $g_1 = 1$  and the claim is proven for $k = 2$.
    Suppose now $g \in \gamma_{k}(G)\cap D_n(G)$ for $k \geq 3$, i.e.
    \begin{equation}
        g = \left(\prod_{\ell=1}^{k-2} g_\ell \right) \cdot g_{k-1} \cdot \left(\prod_{i=k}^{n} g_i\right) = z' \cdot g_{k-1} \cdot z.
    \end{equation}
    Then, as $g \in \gamma_{k}(G)\cap D_n(G) \leq \gamma_{k-1}(G)\cap D_n(G)$, by the inductive hypothesis, one has that $g_\ell = 1$ for all $1 \leq \ell \leq k-2$, i.e. $z' = 1$. As $z \in \gamma_{k}(G)$ one must have $g_{k-1} = 1$.
\end{proof}

\begin{prop}
    Let $G$ be a $\gamma$-free group. The image of $\grM_\bullet(G)$ under the canonical map $\rho_{G, \bullet}$ is isomorphic to:    \begin{equation}\label{eq:im.rho}
        \bigoplus_{n \geq 1} \gamma_n(G)/\gamma_n(G)^p\gamma_{n+1}(G).
    \end{equation}
\end{prop}
\begin{proof}
    The kernel of $\rho_{G, n}$ is equal to \begin{equation} 
    \ker(\rho_{G,n}) = (\gamma_n(G) \cap D_{n+1}(G))/{\gamma_{n+1}(G)}.\end{equation} 
    Applying Lemma \ref{lem:Dn.itsc.gamma}, one has the following 
    \begin{equation}\label{eq:gamma.itsc}
        \gamma_n(G) \cap D_{n+1}(G) = D_{n+1, n}(G) = \gamma_n(G)^{p^{j(n+1, n)}}\gamma_{n+1}(G)^{p^{j(n+1, n+1)}} 
    \end{equation}
    One has that $j(n+1, n) = 1$ and $j(n+1, n+1) = 0$. So, by the first and third homomorphism theorems, one has that
    \[
        {\rm im}(\rho_{G, n}) \cong \gamma_n(G)/(\gamma_n(G)^p\gamma_{n+1}(G)) 
    \]
    for all $n \geq 1$, which proves the claim.
\end{proof}

\begin{prop}\label{prop:basis}
    Let $G$ be a nilpotent $\gamma$-free \textup{(}pro-$p$\textup{)} group of nilpotency class $c$. For every $n \geq 1$, one can choose a $\Z$-basis \textup{(}resp. $\Z_p$-basis\textup{)}  \[
        \{\,x_{n, \lambda} \gamma_{n+1}(G) \mid \lambda \in \Lambda_n\,\},
    \] of the free abelian \textup{(}pro-$p$\textup{)} group $\grM_n(G) = \gamma_n(G)/\gamma_{n+1}(G)$,
    where $x_{n, \lambda}\in\gamma_n(G),\, \lambda \in \Lambda_n$ with totally ordered indexing set $\Lambda_n$, such that the following hold:
    \begin{enumerate}
        \item For every $g \in G$ there exist unique elements $ \alpha_{n, \lambda}(g)$ in $\Z$ \textup{(}resp. $\Z_p$\textup{)} with $n \in \{1, \dots, c\}$ and $\lambda \in \Lambda_n$ such that: 
        \begin{equation}
            g = \prod_{n = 1}^c \prod_{\lambda \in \Lambda_n} x_{n, \lambda}^{\alpha_{n, \lambda}(g)}
        \end{equation}
        where the product is taken with respect to increasing lexicographic order \textup{(}i.e. $x_{n, \lambda}$ precedes $x_{n', \lambda'}$ if 
        either $n < n'$ or $n = n'$ and $\lambda < \lambda'$ in $\Lambda_n$\textup{)}.

         \item An element $g \in G$ belongs in $D_n(G)$ if, and only if $\alpha_{i, \lambda}(G) \in p^{j(n, i)}\Z$ \textup{(}resp. $\alpha_{i, \lambda}(g) \in p^{j(n, i)}\Z_p$\textup{)} for all $i, j \geq 0$.  
    \end{enumerate}
\end{prop}
\begin{proof}
    (1) Let $\vert\argu \vert : G \to \mathbb N_0$ be the valuation associated with the lower central series of $G$. We set $|g| = c+1$ if $g$ is the identity element of $G$, while for $g \neq 1$, one has $|g| = i$ if and only if $g \in \gamma_i(G) \setminus \gamma_{i+1}(G)$. If $g \neq 1$, then $1 \leq |g| \leq c$. In order to prove the claim we proceed by induction on the nilpotency class $c$ of $G$. If $G$ is of nilpotency class $c$ we choose $x_{k,\lambda}$, $1\leq k\leq c-1$, as preimage
    of $x_{k,\lambda}^\prime$
    under $G\to G/\gamma_c(G)$, where $\{\,x_{k,\lambda}^\prime\mid 1\leq k\leq c-1,\,\lambda\in\Lambda_k\,\}$ is a set of elements for which claim (1) holds for $G/\gamma_c(G)$.
    As $\gamma_c(G) \leq Z(G)$ is a finitely generated free abelian (pro-$p$) group, there exist elements $\{\,x_{c, \lambda}\in \gamma_c(G) \mid \lambda \in \Lambda_c\,\}$ such that for all $g \in \gamma_c(G)$ one has 
    \begin{equation}\label{eq:g.gamma.c}
        g = \prod_{\lambda \in \Lambda_c} x_{c, \lambda}^{\alpha_{c, \lambda}(g)}
    \end{equation}
    for  
    uniquely defined elements $\alpha_{c, \lambda}(g)\in\Z$
    (resp. $\Z_p$). We now claim that assertion (1) holds for the set of elements $\{\,x_{k, \lambda} \mid 1\leq k \leq c,\, \lambda \in \Lambda_k\,\}$. Let $g \in G$, then there exist uniquely defined elements $\alpha_{k, \lambda}(g\gamma_c(G))$ satisfying 
    \begin{equation}\label{eq:g.quot}
        g\gamma_c(G) = \prod_{k=1}^{c-1}\prod_{\lambda \in \Lambda_k} (x_{k, \lambda}^\prime)^{\alpha_{k,\lambda}(g\gamma_c(G))}.
    \end{equation}
    By construction, 
    \begin{equation}\label{eq:y}
        y =  \left(\prod_{k=1}^{c-1}\prod_{\lambda \in \Lambda_{k}} x_{k, \lambda}^{\alpha_{k, \lambda}(g\gamma_c(G))}\right)^{-1} \cdot g \in \gamma_{c}(G),
    \end{equation}
    thus by equation \eqref{eq:g.gamma.c} and putting $\alpha_{k, \lambda}(g) = \alpha_{k, \lambda}(g\gamma_c(G))$ for all $k \in \{1, \dots, c-1\}$ and all $\lambda \in \Lambda_k$, the claim follows. 
\medskip \\
(2) By Theorem \ref{th:jennings}, it sufficies to show that an element $g \in D_n(G)$ has the desired property. We proceed by induction on the nilpotency class $c$ of $G$. For $c = 1$, $G$ is an abelian, torsion-free (pro-$p$) group, so the set of elements $\{x_{1, \lambda}\}_{\lambda \in \Lambda_1}$ is a free generating system for $G$. As $D_n(G) = G^{p^{j(n, 1)}}$, there is nothing to prove. Let $g \in D_{n}(G)$. The element $g\gamma_c(G) \in D_n(G)/\gamma_c(G)$ can be written as in \eqref{eq:g.quot}. By the inductive hypothesis, one has that $p^{j(n, i)}$ divides the coefficient $\alpha_{i, \lambda}(g\gamma_c(G)) = \alpha_{i, \lambda}(g)$ for all $1 \leq i \leq c-1$ and $\lambda \in \Lambda_i$. As $x_{i, \lambda}^{\alpha_{i, \lambda}(g)} \in D_n(G)$, the element $y$, defined as in \eqref{eq:y}, is contained in $D_{n, c}(G)$. By Lemma 3, one has
\[
    D_{n, c}(G) = \gamma_c(G) \cap D_n(G) = \gamma_{c}(G)^{p^{j(n, c)}}
\]
Thus, $p^{j(n, c)}$ divides $\alpha_{c, \lambda}(g)$ for all $\lambda \in \Lambda_c$.

\end{proof}

\noindent
We are now able to prove the main theorem of this paper.
\begin{theoremB}\label{th:hat.theta.iso}
	Let $G$ be a finitely generated $\gamma$-free (pro-$p$) group, then the natural map 
\[
	\hat\theta_{G, \bullet} : (\grM_\bullet(G) \otimes \FF_p)^{[p]} \to \grZ_\bullet(G)
\]
	is an isomorphism of restricted $\FF_p$-Lie algebras.
\end{theoremB} 
\begin{proof}
	By Lemma \ref{lem:hat.theta.epi}, it is sufficient to prove that $\hat\theta_{G, n}$ is injective for every $n \geq 1$.
    Recall that any nilpotent quotient $G/\gamma_{c+1}(G)$ of $G$ is also $\gamma$-free by Lemma \ref{lem:gamma.free.equiv}. Thus, we may assume that $G$ is nilpotent of class $c \geq 1$. Using Proposition \ref{prop:basis} (1), we can fix a basis $\{x_{n, \lambda} \gamma_{n+1}(G)\}_{\lambda \in \Lambda_n}$ for each homogeneous component $\grM_n(G)$ of the Magnus $\Z$-Lie algebra. Write $R_\bullet$ for the restrictification of the $\FF_p$-Lie algebra $\grM_\bullet(G) \otimes\FF_p$. As in Remark \ref{rem:pbw.restricted}, the set $\{ (x_{i, \lambda}\gamma_{i+1}(G))^{[p]^k} \mid i, k \geq 0, \lambda \in \Lambda_i \}$ is an $\FF_p$-basis for $R_\bullet$. \\
    Using the PBW theorem for restricted universal enveloping algebras of a restricted $\FF_p$-Lie algebra, we can write a basis for the $n$-th homogeneous component of $R_\bullet$, i.e. for all $n \geq 1$, the set
	\begin{equation}
		\left\{ (x_{i, \lambda} \gamma_{i+1}(G))^{[p]^j} \mid i, j \geq 1 \text{ such that } ip^j = n, \, \lambda \in \Lambda_i \right\}
	\end{equation}
    is a $\FF_p$-basis for $R_n$. Let $y \in R_n$, one has
    \begin{equation} 
        y = \sum_{\substack{i, j \geq 0,\, ip^j = n\\\lambda \in \Lambda_i}} \mu_{i, \lambda}\cdot (x_{i, \lambda}\gamma_{i+1}(G))^{[p]^j} 
    \end{equation}
    for some $\mu_{i, \lambda} \in \FF_p$. Consider the set 
    \[
        S(y) = \{ ((i, j), \lambda_i)\ \mid\ i, j \geq 0,\, ip^j = n, \lambda_i \in \Lambda_i \text{ such that } \mu_{i, \lambda} \neq 0 \},
    \]
    clearly $y = 0$ in $R_n$ if and only if $S(y) = \emptyset$. Let $\mathcal O$ be equal to $\Z$ (resp. $\Z_p$ depending on whether $G$ is a discrete group or a pro-$p$ group) and let $\overline{\phantom x} :  \FF_p \to \mathcal O$ be a section for the canonical projection of $\mathcal O$ onto $\FF_p$. The image of such an element $y \in R_n$ via the homomorphism $\hat\theta_{G, n}$ can be written (in multiplicative notation) as:
    \begin{equation}
        \hat\theta_{G, n}(y) = \prod_{\substack{i, j \geq 0,\, ip^j = n \\ \lambda \in \Lambda_i}} x_{i, \lambda}^{\overline{\mu_{i, \lambda}}p^j}D_{n+1}(G).
    \end{equation}
    
   By Proposition \ref{prop:basis} (2), one has $\prod_{((i,j),\lambda_i)\in S(y)} x_{i, \lambda}^{\overline{\mu_{i, \lambda}}p^j}\in D_n(G)\setminus D_{n+1}(G)$ for all $y\not=0$ showing that $\hat\theta_{G,n}(y)\in D_n(G)\setminus \{0\}$. Thus $\hat\theta_{G,n}$
    is injective.

\end{proof}

\section{Applications}
As was pointed out to us by an anonymous referee, $\LL^{[p]}$ admits the \qt{same} presentation as $\LL$ in the category of restricted $\FF$-Lie algebras.
\begin{remark}\label{rem:pres.restricted}
    Let $L(X)$ be the free $\FF$-Lie algebra over the set $X$ and let $L^{[p]}(X)$ be the free restricted $\FF$-Lie algebra over the same set. Note that $L^{[p]}(X)$ coincides with the restrictification $L(X)^{[p]}$ of $L(X)$ and thus $L(X)$ can be considered an $\FF$-Lie subalgebra of $L^{[p]}(X)$.
    Consider $\LL = L(X)/\langle R\rangle$, where $R$ is a subset of $L(X)$. Then $L^{[p]}(X)/\langle\!\langle R\rangle\!\rangle$, where $\langle\!\langle R\rangle\!\rangle$ is the restricted Lie ideal generated by $R$, satisfies the same universal property as $\LL^{[p]}$ and thus is isomorphic to $\LL^{[p]}$.
\end{remark}

In this section we list some examples of groups $G$ for which Theorem B may be applied to, that is, groups satisfying
\[
	(\grM_\bullet(G) \otimes \FF_p)^{[p]} \cong \grZ_\bullet(G).
\] 
By Remark \ref{rem:pres.restricted}, this allows to calculate a presentation of the Zassenhaus restricted $\FF_p$-Lie algebra of such a group $G$ knowing the presentation of its Magnus $\Z$-Lie algebra.

\subsection{Free groups}
A theorem attributed to Magnus states that if $F_n$ is the free group of rank $n$, then the associated Lie algebra $\grM_\bullet(F_n)$ is the free $\Z$-Lie algebra on $n$ generators, in particular $F_n$ is $\gamma$-free. The ranks of each homogeneous component, $M_n(k) = {\rm rk}(\grM_k(F_n))$, were computed by Hall and Witt (see \cite{Bou89}) and they satisfy the following formula:
\begin{equation}\label{eq:hall.witt.ranks}
	M_n(k) = \frac{1}{k} \sum_{d | k} \mu(d) n^{k/d}.
\end{equation}
where $\mu$ is the M\"obius function.
\subsection{One relator groups}
A finitely generated one relator group is a quotient $G \cong F/R$ of a finitely generated free group $F$ by the normal closure of a cyclic subgroup $R = \langle r \rangle^F \unlhd F$, the word $r \in F$ is the defining relation of $G$. For any word $r \in F$, define its \emph{weight} to be the largest integer $w(r) > 0$ such that $r \in \gamma_{w(r)}(G) \setminus \gamma_{w(r)+1}(G)$. As $F$ is residually nilpotent, $r \neq 1$ implies $w(r) < \infty$. Following \cite{Lab70}, we will say that $r$ is \emph{primitive} if it is not a proper power modulo $\gamma_{w(r)+1}(F)$. The following theorem is due to Labute. 
\begin{theorem}[{\cite{Lab70}}]
	Let $G = F/R$ where $F$ is a finitely generated free group and $R$ is the normal closure of the subgroup generated by $r \in F$, a primitive word. Let $\bar{r}$ be the image of $r$ in $\grM_{w(r)}(F)$ and let $\mathfrak r$ be the Lie-ideal generated by $\bar{r}$. Then, there is an isomorphism of graded $\Z$-Lie algebras: $$\grM_\bullet(G) \cong \grM_\bullet(F)/\mathfrak r.$$ 
\end{theorem}

As a consequence, any such one-relator groups $G$ is $\gamma$-free, and the ranks of $\grM_\bullet(M)$ can be computed via a formula analogous to Equation \eqref{eq:hall.witt.ranks}. \\
\noindent
An important class of examples of $\gamma$-free one relator groups is given by the fundamental groups of orientable, compact surfaces. If $S_g$ is an orientable, compact surface of genus $g$, its fundamental group is given by 
\begin{equation}
    \pi_1(S_g) = \langle a_1, b_1, \dots, a_g, b_g \mid [a_1, b_1][a_2, b_2] \cdots [a_g, b_g]  \rangle.
\end{equation}
\subsection{Free products and Magnus groups}
Recall that a group $G$ is said to be \emph{residually nilpotent} if the intersection of all the terms of the lower central series, $\bigcap\nolimits_{n \geq 1} \gamma_n(G)$, 
is the trivial subgroup. Groups which are both residually nilpotent and $\gamma$-free, often referred to as \emph{Magnus groups}, have been the subject of many investigations in the literature (see for example \cite{Bau67}). We record the following theorem due to Shmelkin.
\begin{theorem}[{\cite{Shm70}}]\label{th:magnus.groups}
    The free product of Magnus groups is again a Magnus group, moreover one has that the Magnus $\Z$-Lie algebra functor restricted to the class of Magnus group commutes with coproducts.
\end{theorem}
\noindent
This result allows us to construct, via free products, a wide variety of examples of (residually nilpotent) groups for which Theorem B holds.

\subsection{Almost direct products of free groups}

A semidirect product $G = N \rtimes H$ for which the inclusion $[N, H] \leq [N, N]$ holds, is said to be \emph{almost-direct}. Equivalently, such a group is almost-direct if the conjugation action of $H$ on $N$ descends to a trivial action on the abelianization $N/[N, N]$. The Magnus $\Z$-Lie algebra of an almost direct product was studied in detail by Falk and Randell:
\begin{theorem}[\cite{FR85}]\label{th:falk.randell}
	Let $G = N \rtimes H$ be an almost direct product. Then for all $n \geq 1$ one has the following isomorphisms of abelian groups:
\[
		\grM_n(G) \cong \grM_n(N) \oplus \grM_n(H)
\]
\end{theorem} 
\noindent
In particular almost direct products of $\gamma$-free groups are again $\gamma$-free, hence Theorem B can be applied.

\subsection{Hyperplane and Toric arrangements}
A \emph{complex hyperplane arrangement} in $\mathbb C^\ell$ is a finite collection $\mathcal A = \{H_1, \dots, H_n\}$ of codimension $1$ vector subspaces of $\mathbb C^\ell$. The complement space $M(\mathcal A) = \mathbb C^\ell \setminus \cup_{H \in \mathcal A} H$ is a connected manifold whose topology has been the subject of extensive investigations. Many interesting questions can be asked about its cohomology ring $H^\bullet(M(\mathcal A), \Z)$ and its fundamental group $G(\mathcal A) = \pi_1(M(\mathcal A), *)$. 
The \emph{intersection lattice} of $\mathcal A$, denoted $\mathbb L(\mathcal A)$, is defined as the set of all the possible nonempty intersections of the hyperplanes of $\mathcal A$, ordered by reverse inclusion. The combinatorial information contained in the lattice $\mathbb L(\mathcal A)$ can be used to define the so-called Orlik-Solomon algebra $A^\bullet(\mathcal A)$ which, by the classic Brieskorn-Orlik-Solomon theorem \cite{OS80}, is isomorphic to the cohomology ring $H^\bullet(M(\mathcal A),\Z)$.
A hyperplane arrangement $\mathcal A$ is said to be \emph{supersolvable} if its intersection lattice is a supersolvable lattice i.e. if it admits a maximal chain of modular elements (see \cite{Sta72}). This purely combinatorial condition has surprising implications in the topology of the arrangement. In particular for a supersolvable arrangement $\mathcal A$ in $\mathbb C^\ell$, it implies that the complement manifold $M(\mathcal A)$ sits on top of a tower of fibrations:
\[
	M_\ell = M(\mathcal  A) \longrightarrow M_{\ell-1} \longrightarrow \cdots \longrightarrow M_1 
\]
where, for $1 \leq k \leq \ell-1$, the space $M_k = M(\mathcal A_k)$ is the complement of a supersolvable arrangement $\mathcal A_k$ in $\mathbb C^k$. Moreover, the generic fiber of each fibration is homeomorphic to a copy of $\mathbb C$ minus a finite number of points. As a consequence, the fundamental group of a supersolvable arrangement splits as an iterated semidirect product of free groups of finite rank: 
\begin{equation}\label{eq:ss.arr}
    G(\mathcal A) = F_{d_\ell} \rtimes (F_{d_{\ell-1}} \rtimes ( \cdots \rtimes (F_{d_2} \rtimes F_{d_1} )))
\end{equation} 
where the ranks (usually called the \emph{exponents} of the arrangement) are given by $d_k = |\mathcal A_k| - |\mathcal A_{k-1}|$ 
for $2 \leq k \leq \ell$ and $d_1 = |\mathcal A_1|$. Since the semidirect products in \eqref{eq:ss.arr} are almost-direct products by \cite[Proposition 2.5]{FR85}, the fundamental group of a supersolvable hyperplane arrangement complement is $\gamma$-free.

\subsubsection{Pure braid groups}
One of the most important examples of supersolvable hyperplane arrangement is given by the \emph{braid arrangement} \[
    \mathcal Br_n = \{ H_{ij} \mid 1 \leq i < j \leq n+1\}
\]
in $\mathbb C^{n+1}$ (for $n \geq 2$), where $H_{ij}$ is the hyperplane defined by the equation $x_j - x_i = 0$. The complement of this arrangement can be identified with the \emph{configuration space} of $n$ distinct points in $\mathbb C$, so its fundamental group
is isomorphic to the pure braid group on $n$ strands, i.e. $G(\mathcal Br_n) \cong {\rm PB}_n$. Thus, the pure braid group on $n$ strands is $\gamma$-free. A presentation of the Magnus $\Z$-Lie algebra was given by Kohno in \cite{Koh85}, namely $\grM_\bullet({\rm PB}_n)$ is isomorphic to the quotient of the free $\Z$-Lie algebra over the set $\{X_{ij} \mid i \leq i < j \leq n+1\}$ by the ideal generated by the elements:
\begin{itemize}
    \item $[X_{ik}, X_{ij}+X_{jk}]$ and $[X_{jk}, X_{ij}+X_{ik}]$ for all $1 \leq i < j < k \leq n+1$;
    \item $[X_{ir}, X_{js}]$ for all $i < r$, $j < s$, $i < j$ and $r \neq j$, $r \neq s$.
\end{itemize}
This Lie algebra can be identified with the \emph{holonomy Lie algebra} of the associated hyperplane arrangement.

\subsubsection{Strictly supersolvable toric arrangements}
The theory of hyperplane arrangement naturally generalizes to a setting where the ambient space $\mathbb C^\ell$ is replaced by the complex torus $\mathbb T^\ell = (\mathbb C^\times)^\ell$ and the linear hyperplanes are replaced by kernels of monomial characters $\chi : \mathbb T^\ell \to \mathbb C^\times$ called \emph{hypertori}. A \emph{toric arrangement} in $\mathbb T^\ell$ is a finite collection $\mathcal B$ of hypertori. Given a toric arrangement $\mathcal B$ one can study, similarly to the case of hyperplane arrangements, the complement manifold $M(\mathcal B) = \mathbb T^\ell \setminus \cup_{K \in \mathcal B}K$, its cohomology ring $H^\bullet(M(\mathcal B), \Z)$ and its fundamental group $G(\mathcal B) = \pi_1(M(\mathcal B), *)$. \\
The combinatorial object associated to a toric arrangement $\mathcal B$ is the \emph{poset of layers} $\mathbb L(\mathcal B)$ whose elements are the connected components of all the possible nonempty intersections of the hypertori of $\mathcal B$. 
The theory of toric arrangements shows some significant differences from its linear counterpart because of the different topology of the ambient space: two connected hypertori may have a non-connected intersection and even a single hypertorus may consist of several connected components. For this reason, the poset of layers is not in general a lattice, but rather a locally geometric poset. Extending the definition of supersolvable lattices allows to define the class of \emph{supersolvable toric arrangements}. This definition and the resulting theory is developed in \cite{BD24}. \\
\noindent
By \cite[Theorem 3.4.3]{BD24}, the fundamental group $G(\mathcal B)$ of the complement of a supersolvable toric arrangement $\mathcal B$ is an iterated semidirect product of free groups of finite rank. A further combinatorial restiction on the poset of layers called \emph{strict} supersolvability ensures that the semidirect products are almost-direct (see \cite[Theorem 5.3.4 and Theorem 5.3.10]{BD24}). We have then that the fundamental groups of complements of strictly supersolvable toric arrangements are $\gamma$-free groups.

\subsection{Uniformly powerful pro-$p$ groups}
\label{sss:unif}
A finitely generated pro-$p$ group $G$ is said to be uniform,
if it is torsion-free and satisfies $\gamma_2(G)\leq G^p$ if $p$ is odd, and $\gamma_2(G)\leq G^4$ if $p=2$.
It is well-known that for these groups one has
$D_n(G,\FF_p)=G^{p^n}$ and that the $p$-map $(\argu)^{[p]}\colon\gr^{\mathcal Z}_n(G) \to\gr^{\mathcal Z}_{n+1}(G)$ (induced by the group theoretic $p$-power map) is injective (see \cite{DDSMS}). This yields the description of the Magnus-Zassenhaus kernel given in Corollary~\ref{cor2}.
By definition, neither non-trivial groups of exponent $p$ nor non-trivial uniform pro-$p$ groups can be $\gamma$-free.
Nevertheless, there are examples of $p$-adic analytic $\gamma$-free pro-$p$ groups, e.g.,
\begin{equation}
\label{eq:utrZp}
G = \Bigg\{\left(\begin{matrix} 1&x & z \\ 0 & 1&y\\
0&0&1\end{matrix}\right)\ \Bigg\vert\ x,y,z\in\Z_p\,\Bigg\}
\end{equation} 
is a $p$-adic analytic $\gamma$-free
pro-$p$ group of nilpotency class $2$.

\subsection{A class of mild non $\gamma$-free pro-$p$ groups}
\label{sss:mild}
For a positive integer $n\geq 3$ and an odd prime $p$, let
\begin{equation}
    \label{eq:labu1}
    G_n = \langle\, x_0, \dots, x_{n-1} \mid x_k^p = [x_{k-1}, x_k],\, k \in \Z/n\Z \,\rangle_{\text{pro-}p}
\end{equation}
These groups were studied by J. Labute  in \cite{Lab06}
There it is shown that the groups $G_n$ are mild pro-$p$ groups for odd integers $n
\geq 5$. More precisely, their Zassenhaus restricted $\FF_p$-Lie algebra
coincides with the restrictification of the Magnus $\Z$-Lie algebra tensored with $\FF_p$ and equals
\begin{equation}
\label{eq:labu2}
    \grZ(G_n)\simeq\grZ(F)/\langle r_1, \dots, r_n \rangle_{\rLie},
\end{equation}
where $F$ is the free group of rank $n$ and $r_k=[x_{k-1},x_k]$ (see \cite{Lab06}).
Therefore, the group $G_n$ satisfies the conclusion of
Theorem B. However, as the abelianzation $G_n/\gamma_2(G_n)$ has $p$-torsion, $G_n$ is not $\gamma$-free. 

\bibliographystyle{plain}
\bibliography{refs}

\end{document}